\documentclass[a4paper, 11pt]{article} 
\usepackage[headings]{fullpage}   
\usepackage{amsfonts}
\usepackage{amsmath} 
\usepackage{amssymb}
\usepackage{graphicx}
\usepackage{amsthm}
\usepackage[dvipsnames,svgnames]{xcolor}
\usepackage[colorlinks=true,linkcolor=RoyalBlue,urlcolor=RoyalBlue,citecolor=PineGreen]{hyperref}
\usepackage[nameinlink]{cleveref}
\usepackage{tikz}
\usetikzlibrary{angles}
\usetikzlibrary{calc}
\usetikzlibrary {intersections,through}
\usepackage{enumerate}
\hypersetup{linktocpage=true}
\newtheorem{theorem}{Theorem}[section]
\newtheorem{lemma}[theorem]{Lemma}
\newtheorem{conjecture}[theorem]{Conjecture}
\newtheorem{proposition}[theorem]{Proposition}
\newtheorem{corollary}[theorem]{Corollary}
\theoremstyle{definition}
\theoremstyle{definition}\newtheorem{example}[theorem]{Example}
\theoremstyle{definition}\newtheorem{remark}[theorem]{Remark}
\theoremstyle{definition}

\newcommand{\rank}{\operatorname{rank}}

\begin{document}

\title{Rigid frameworks with dilation constraints}

\author{Sean Dewar\thanks{
Mathematics, University of Bristol, BS8 1QU, UK. E-mail: sean.dewar@bristol.ac.uk} \and Anthony Nixon\thanks{Mathematics and Statistics, Lancaster University, LA1 4YF, UK. E-mail: a.nixon@lancaster.ac.uk} \and Andrew Sainsbury\thanks{Mathematics and Statistics, Lancaster University, LA1 4YF, UK. E-mail: a.sainsbury2@lancaster.ac.uk}
}

\date{}

\maketitle

\begin{abstract}
We consider the rigidity and global rigidity of bar-joint frameworks in Euclidean $d$-space under additional dilation constraints in specified coordinate directions. In this setting we obtain a complete characterisation of generic rigidity. We then consider generic global rigidity. In particular, we provide an algebraic sufficient condition and a weak necessary condition. We also construct a large family of globally rigid frameworks and conjecture a combinatorial characterisation when most coordinate directions have dilation constraints.
\end{abstract}

\section{Introduction}

A bar-joint \emph{framework} $(G,p)$ is the combination of a (finite, simple) graph $G=(V,E)$ and a map $p:V\rightarrow \mathbb{R}^d$ that assigns positions in Euclidean $d$-space to the vertices of $G$ (and hence straight line segments of given length to the edges). Informally, the framework is \emph{rigid} if its lengths locally determine the shape. That is, if every edge-length preserving continuous motion of the vertices arises from a Euclidean isometry. 

Determining the rigidity or flexibility (non-rigidity) of a framework is a crucial problem in a variety of practical applications from  wireless sensor networks \cite{Sensors} to control of robotic formations \cite{Robotics}, and rigidity theoretic tools have recently been used to impact on diverse mathematical problems such as the lower bound theorem for manifolds \cite{CJT,Kal}. However, it is computationally challenging to determine the rigidity of a given framework (when $d>1$) \cite{abbott}. To get around this issue we will consider generic frameworks; $(G,p)$ is \emph{generic} if the set of coordinates of the points $p(v)$, $v\in V$, are distinct and form an algebraically independent set over $\mathbb{Q}$\footnote{Much weaker variants of genericity are possible but a distraction from the main topic of this paper.}. 

Generically, rigidity depends only on the underlying graph. That is, if a generic framework $(G,q)$ in $\mathbb{R}^d$ is rigid then every generic framework $(G,p)$ in $\mathbb{R}^d$ is rigid. The cornerstone problem of rigidity theory is to determine precisely the class of graphs that are rigid. As far back as Maxwell \cite{Max} necessary conditions were known but these are insufficient for all $d\geq 3$. When $d=1$ it is folklore that a graph is rigid if and only if it is connected. When $d=2$, a celebrated theorem first obtained by Polaczek-Geiringer \cite{PG}, and often referred to as Laman's theorem \cite{laman}, characterises rigid graphs. However for all $d\geq 3$, characterising rigid graphs remains a challenging open problem of key applied and theoretical importance.

Motivated by this, the present article considers the rigidity of Euclidean frameworks under additional ``coordinate dilation'' constraints. These frameworks first arose in the context of frameworks on surfaces \cite{JNglobal} where characterising rigidity with these dilation constraints was important in understanding stress matrices and global rigidity.
This kind of constraint is previously unstudied in the Euclidean context. 
A somewhat related setting occurred recently in the study of unmanned aerial vehicles \cite{capr}. We will describe this at the end of \Cref{sec:char}.

We give precise descriptions of rigidity in $d$-dimensions under dilation constraints linking them to ``ordinary'' rigidity in lower dimensions, and hence giving purely combinatorial characterisations in arbitrary dimension provided there are sufficiently many coordinate dilation constraints. As a consequence (\Cref{cor:union}) we establish that the underlying rigidity matroid of our dilation constraint setting is the union of a smaller dimensional rigidity matroid and the uniform matroid of a specified rank. In this sense our results are similar to recent results obtained in \cite{DK} for cylindrical normed spaces and \cite{SST} for coordinated edge length motions.

We then investigate variants for global rigidity. Global rigidity of bar-joint frameworks asks, more strongly than rigidity, for the given framework to be unique to ambient motions of the space. It follows from a deep result of Gortler, Healy and Thurston \cite{GHT} that, generically, global rigidity depends only on the underlying graph and the combinatorial situation mirrors rigidity. In dimension 1 a graph is generically globally rigid if and only if it is 2-connected, in dimension 2 a combination of results due to Hendrickson \cite{hendrickson}, Connelly \cite{C2005} and Jackson and Jord\'an \cite{J&J} give a complete combinatorial characterisation, whereas  when $d\geq 3$ only some partial results are known.

In \Cref{sec:global} we develop an augmented definition of equilibrium stress and stress matrix and use them to give an analogue of Connelly's sufficient condition \cite{C2005} that applies to global rigidity in $\mathbb{R}^d$ with dilation constraints. We deduce from this that a well known construction operation (1-extension) preserves global rigidity and then discuss necessary conditions. We conclude the paper, in \Cref{sec:conclude}, with two open problems on global rigidity.

\section{Rigidity theoretic preliminaries}

Let $G=(V,E)$ be a graph. Two frameworks $(G,p)$ and $(G,q)$ are said to be \emph{equivalent} if 
\begin{equation}\label{eq:constraints}
    \|p(v) - p(w)\| = \|q(v) - q(w)\| \qquad \text{ for all } vw \in E.
\end{equation}
More strongly, they are \emph{congruent} if \cref{eq:constraints} holds for any pair of vertices $v,w\in V$. 
The framework $(G,p)$ in $\mathbb{R}^d$ is \emph{$d$-rigid} if every equivalent framework $(G,q)$ in a neighbourhood of $p$ in $\mathbb{R}^d$ is obtained from $(G,p)$ by a composition of isometries of $\mathbb{R}^d$. Given that $G$ contains at least $d$ vertices and $(G,p)$ affinely spans $\mathbb{R}^d$, this is equivalent to requiring that $(G,q)$ is congruent to $(G,p)$. Moreover, $(G,p)$ in $\mathbb{R}^d$ is \emph{globally $d$-rigid} if every equivalent framework $(G,q)$ in $\mathbb{R}^d$ is obtained from $(G,p)$ by a composition of isometries of $\mathbb{R}^d$.
A framework is \emph{minimally $d$-rigid} if it is $d$-rigid, but any framework formed by removing edges is not $d$-rigid.

Differentiating the distance constraints given in \cref{eq:constraints}, we obtain the following.
An {\em infinitesimal motion} of $(G, p)$ is a map $\dot
p:V\to \mathbb{R}^d$ satisfying the system of linear equations:
\begin{equation*}
(p(v)-p(w))\cdot (\dot p(v)-\dot p(w)) = 0 \qquad \mbox{ for all $vw \in E$}.
\end{equation*}
The framework $(G,p)$ is \emph{infinitesimally $d$-rigid} if the only infinitesimal motions arise from isometries of $\mathbb{R}^d$.
The {\em rigidity matrix $R (G, p)$} of the framework $(G, p)$ is  the
matrix of coefficients of this system of equations for the unknowns
$\dot p$. Thus $R (G, p)$ is a $|E|\times d|V|$ matrix, in
which: the row indexed  by an edge $vw \in E$ has $p(v)-p(w)$ and
$p(w)-p(v)$ in the $d$ columns indexed by $v$ and $w$ respectively, and zeros elsewhere.

Rigidity and infinitesimal rigidity are linked by the following theorem.

\begin{theorem}[Asimow and Roth \cite{asi-rot}]\label{thm:AR}
For a generic framework $(G,p)$ in $\mathbb{R}^d$ on at least $d$ vertices, the following are equivalent:
\begin{enumerate}
    \item $(G,p)$ is $d$-rigid;
    \item $(G,p)$ is infinitesimally $d$-rigid;
    \item $\rank R(G,p)=d|V|-\binom{d+1}{2}$.
\end{enumerate}
\end{theorem}

For the next result, we recall that a graph $G=(V,E)$ is \emph{$(d,\binom{d+1}{2})$-sparse} if $|E'| \leq d|V'|-\binom{d+1}{2}$ for all subgraphs $(V',E')$ on at least $d$ vertices. A $(d,\binom{d+1}{2})$-sparse graph is $(d,\binom{d+1}{2})$-tight if $|E|=d|V|-\binom{d+1}{2}$.

\begin{lemma}[Maxwell \cite{Max}]
Let $(G,p)$ be minimally $d$-rigid on at least $d$ vertices. Then $G$ is $(d,\binom{d+1}{2})$-tight.
\end{lemma}

The converse to Maxwell's lemma holds when $d\leq 2$ \cite{PG}. However it remains an open problem to determine which $(d,\binom{d+1}{2})$-tight graphs are $d$-rigid in all higher dimensions. One motivation for this paper is that additional understanding of this problem may be developed by exploring the problem with additional constraints.

An \emph{equilibrium stress} of a framework $(G,p)$ in $\mathbb{R}^d$ is a vector in the cokernel of $R(G,p)$. Equivalently, a vector $\omega\in \mathbb{R}^{|E|}$ is an
\emph{equilibrium stress} for $(G,p)$ if, for all
$v\in V$,
\begin{equation*}
\sum_{w\in N_G(v)} \omega_{vw}(p(v)-p(w)) =0,
\end{equation*}
where $N_G(v)$ denotes the neighbour set of $v$.
Let $n=|V|$. The
\emph{stress matrix} $\Omega(\omega)$ is a symmetric $n\times n$-matrix in which the rows and columns are indexed by the vertices and in which the off
diagonal entry in row $v$ and column $w$ is $-\omega_{vw}$, and the
diagonal entry in row $v$  is $\sum_{w\in V} \omega_{vw}$. (Here $\omega_{vw}$ is taken to be equal to zero if $vw\not\in E$.)
Equivalently,
the stress matrix $\Omega(\omega)$ is the Laplacian matrix of the weighted graph $(G,\omega)$.

\section{Infinitesimal \texorpdfstring{\lowercase{$(d,k)$}}{\lowercase{(d,k)}}-rigidity}

Let $(G,p)$ be a framework in $\mathbb{R}^d$.
Fix some $k \in \{1,\ldots,d-1\}$.
For each coordinate $i \in \{1,\ldots,d\}$,
set $p_i :  V \rightarrow \mathbb{R}$ be the restriction of $p$ to the $i$-th coordinate.
Using this notation, we see that $p(v) = (p_1(v), \ldots p_d(v))$ for every vertex $v \in V$.
With this, we define the map
\begin{align*}
    \tilde{p} :V \rightarrow \mathbb{R}^{d-k}, ~ v \mapsto (p_1(v), \ldots, p_{d-k}(v)).
\end{align*}
Two frameworks $(G,p)$ and $(G,q)$ in $\mathbb{R}^d$ are \emph{$(d,k)$-equivalent} if they are equivalent and for each $i \in \{d-k+1, \ldots, d\}$ there exists a scalar $\alpha_i \neq 0$ such that $p_i = \alpha_i q_i$.
If we fix a vertex $v_0 \in V$ where $p_i(v_0) \neq 0$ for each $i \in \{d-k+1, \ldots, d\}$,
this can be represented by the following constraint system (under the assumption that $q$ is a realisation with $q_i(v_0) \neq 0$ for $i \in \{d-k+1, \ldots, d\}$):
\begin{align}
    \|p(v) - p(w)\| &= \|q(v) - q(w)\|  \qquad \text{ for all } vw \in E, \label{eqn1}\\
    \frac{p_i(v)}{p_i(v_0)} &= \frac{q_i(v)}{q_i(v_0)}  \qquad \text{ for all } v \in V \setminus \{v_0\} \text{ and } i \in \{d-k+1, \ldots, d\}.\label{eqn2}
\end{align}

We note that the precise choice of constraints in \cref{eqn2} was made for convenience. For most frameworks, constraints corresponding to any connected graph would create precisely the same constraint system as we show in the following simple lemma.

\begin{lemma}
Let $G = (V,E)$, $H = (V,F)$ be two graphs with the same vertex set,
and let $(G,p),(G,q)$ be frameworks in $\mathbb{R}^d$ where $p_i(v) \neq 0$ and $q_i(v) \neq 0$ for all $v \in V$ and $i \in \{d-k+1, \ldots, d\}$.
If $H$ is connected,
then \cref{eqn2} holds if and only if 
\begin{equation}
    \frac{p_i(v)}{p_i(w)} = \frac{q_i(v)}{q_i(w)} \qquad \text{ for all } vw \in F \text{ and } i \in \{d-k+1, \ldots, d\}.\label{eqn2star}
\end{equation}
\end{lemma}

\begin{proof}
Suppose \cref{eqn2} holds.
Choose any $vw \in F$ and any $i \in \{d-k+1, \ldots, d\}$.
Then for any $i \in \{d-k+1, \ldots, d\}$, we have $$\frac{p_i(v)}{p_i(w)} = \frac{p_i(v)}{p_i(v_0)} \cdot \frac{p_i(v_0)}{p_i(w)} =  \frac{q_i(v)}{q_i(v_0)} \cdot \frac{q_i(v_0)}{q_i(w)} = \frac{q_i(v)}{q_i(w)}.$$ 
Hence, \cref{eqn2star} holds.

Suppose \cref{eqn2star} holds.
Choose any $v \in V \setminus \{v_0\}$ and any $i \in \{d-k+1, \ldots, d\}$.
Then since $H$ is connected, there exists a path from $v$ to $v_0$ in $H$, say $(v, v_1, v_2, \ldots, v_t,v_0)$.
For $i \in \{d-k+1, \ldots, d\}$ we have
\begin{align*}
    \frac{p_i(v)}{p_i(v_0)} = \frac{p_i(v)}{p_i(v_1)} \frac{p_i(v_1)}{p_i(v_2)} \ldots \frac{p_i(v_{t})}{p_i(v_0)} = \frac{q_i(v)}{q_i(v_1)} \frac{q_i(v_1)}{q_i(v_2)} \ldots \frac{q_i(v_{t})}{q_i(v_0)} = \frac{q_i(v)}{q_i(v_0)}. 
\end{align*}
Hence, \cref{eqn2} holds.
\end{proof}

The Jacobian derivative of these constraints\footnote{Technically speaking, we are actually taking the derivative of the concatenation of the constraints in \cref{eqn1} after squaring then halving both sides and the constraints given by \cref{eqn2}.} is the $(|E|+k(|V|-1)) \times d|V|$ matrix
\begin{align*}
    J_{v_0}(G,p) =
    \begin{bmatrix}
        R(G,\tilde{p}) & R(G,p_{d-k+1}) & \cdots & R(G,p_{d}) \\
        \mathbf{0} & M_{d-k+1} &   & \mathbf{0} \\
        \vdots &       &  \ddots    &      \\
        \mathbf{0} & \mathbf{0}  &  & M_d
    \end{bmatrix},
\end{align*}
where $R(G,\tilde{p})$ and $R(G,p_i)$, for $d-k+1\leq i \leq d$, are the rigidity matrices of the $(d-k)$-dimensional framework $(G,\tilde{p})$ and the 1-dimensional frameworks $(G,p_i)$ respectively,
and $M_i$ is the matrix with rows labelled by $V\setminus\{v_0\}$, columns labelled by $V$ and entries
\begin{align*}
    M_i(v,w) =
    \begin{cases}
        1/p_i(v_0) &\text{if $w = v$},\\
        -p_i(v)/p_i(v_0)^2 &\text{if $w = v_0$},\\
        0 &\text{otherwise}.
    \end{cases}
\end{align*}

\begin{example}
Let $d=2$, $k=1$, $G=K_3$ and $V(K_3)=\{v_0,v_1,v_2\}$. Then $i=d-k+1=2$, and for illustration we put $\tilde p(v_i)=x_i$, $p_2(v_i)=y_i$, so that $J_{v_0}(K_3,p) = $
\[ 
\begin{pmatrix} 
x_0-x_1 & x_1-x_0 & 0 &y_0-y_1 & y_1- y_0 & 0 \\
x_0-x_2 & 0 & x_2-x_0 &y_0-y_2 & 0 & y_2-y_0 \\
0 & x_1-x_2 & x_2-x_1 &0 & y_1-y_2 & y_2-y_1 \\
0&0&0&-\frac{y_1}{y_0^2}&\frac{1}{y_0}&0\\
0&0&0&-\frac{y_2}{y_0^2}&0& \frac{1}{y_0}
\end{pmatrix}. \]
\end{example}

We now define $(G,p)$ to be \emph{infinitesimally $(d,k)$-rigid} if and only if either $G$ is a complete graph and the set $\{p(v) : v \in V\}$ has affine dimension $\min \{ d, |V|-1\}$ or $G$ is not complete and $\rank J_{v_0}(G,p) = d|V| - \binom{d-k+1}{2}$.
We define a graph $G$ to be \emph{$(d,k)$-rigid} if there exists a realisation $p:V \rightarrow \mathbb{R}^d$ where $(G,p)$ is infinitesimally $(d,k)$-rigid.
Equivalently,
$G$ is $(d,k)$-rigid if every generic $(G,p)$ in $\mathbb{R}^d$ is infinitesimally $(d,k)$-rigid.
Suppose $J_{v_0}(G,p)$ has linearly independent rows and $\rank J_{v_0}(G,p) = d|V| - \binom{d-k+1}{2}$. Then 
\begin{equation*}
    d|V| - \binom{d-k+1}{2}=|E|+k(|V|-1)=(d-k)|V|- \binom{d-k+1}{2} + k,
\end{equation*}
and infinitesimal $(d,k)$-rigidity is lost if an edge is removed.
Hence we say that a graph $G$ is \emph{minimally $(d,k)$-rigid} if it is $(d,k)$-rigid and $|E|=(d-k)|V| - \binom{d-k+1}{2} + k$. 

As in the standard theory of bar-joint rigidity one can define continuous $(d,k)$-rigidity. In the generic case, using the technique of Asimow and Roth \cite{asi-rot}, one can show that continuous $(d,k)$-rigidity and infinitesimal $(d,k)$-rigidity coincide for generic frameworks. Hence we drop infinitesimal in what follows and refer simply to $(d,k)$-rigidity.

\begin{example}
\label{ex:c4}
Suppose $d=2$, $k=1$ and $G=C_4$ be the cycle with vertex set $\{v_1,v_2,v_3,v_4\}$ and edge set $\{v_1v_2,v_2v_3,v_3v_4,v_1v_4\}$. Then putting $p(v_1)=(1,1)$, $p(v_2)=(2,1)$, $p(v_3)=(2,2)$ and $p(v_4)=(1,2)$ gives a framework $(C_4,p)$ such that $\rank J_{v_0}(C_4,p)=7=2|V|-\binom{2}{2}$ (see \Cref{fig:basic}). Since $|E|=4=1\cdot 4-\binom{2}{2}+1$, $(C_4,p)$ is minimally $(2,1)$-rigid. Since rank is maximised at generic configurations, the same conclusion holds for any generic framework $(C_4,q)$. 
Suppose on the other hand that $d=3,k=2$ and $G=C_4$. Then $\rank J_{v_0}(C_4,p)\leq 10 < 3|V|-1=11$ so $(C_4,p)$ is not (infinitesimally) $(3,2)$-rigid for any $p$.
\end{example}

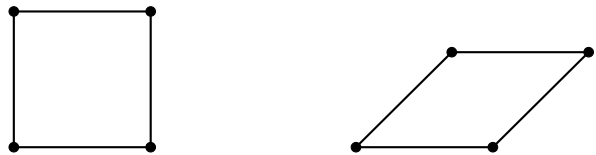
\begin{figure}[h]
\begin{center}
\begin{tikzpicture}[scale=0.9]
\begin{scope}
\filldraw (0,0) circle (2pt);
\filldraw (2,0) circle (2pt);
\filldraw (0,2) circle (2pt);
\filldraw (2,2) circle (2pt);


 \draw[black,thick]
  (0,0) -- (2,0) -- (2,2) -- (0,2) --  (0,0);
  
\end{scope}

\begin{scope}[shift={(5,0)}, scale=2]
\filldraw (0,0) circle (1pt);
\filldraw (1,0) circle (1pt);
\filldraw (0.7,0.7) circle (1pt);
\filldraw (1.7,0.7) circle (1pt);


 \draw[black,thick]
  (0,0) -- (1,0) -- (1.7,0.7) -- (0.7,0.7) --  (0,0);
  
\end{scope}

\end{tikzpicture}
\end{center}
\caption{The 2-dimensional framework described in \Cref{ex:c4} is depicted on the left. This framework is not 2-rigid since there is a non-trivial continuous deformation taking it to the framework on the right. Nevertheless the framework is $(2,1)$-rigid since the dilation constraints in the $y$-coordinates prevent any nontrivial motion. The intuition behind this is to first note that translation in the $y$-direction and rotation evidently break the dilation constraints. Consider now the nontrivial motion depicted. 
The top left vertex follows the path $\theta \mapsto (1 + \sin \theta, 1 + \cos \theta)$ and the top right vertex follows the path $\theta \mapsto (2 + \sin \theta, 1 + \cos \theta)$.
As the bottom two vertices have $y$-coordinate 1,
the dilation constraints require that the $y$-coordinates of the top two vertices -- both of which are $1 + \cos \theta$ -- are constant during the motion,
a clear contradiction.
}
\label{fig:basic}
\end{figure}

\section{Characterising \texorpdfstring{\lowercase{$(d,k)$}}{\lowercase{(d,k)}}-rigidity}
\label{sec:char}

To characterise $(d,k)$-rigidity we will first show that there is a more convenient matrix representation. Recall that the usual $d$-dimensional (squared) rigidity map $f_{G,d}$  is the map 
\begin{equation*}
    f_{G,d}:\mathbb{R}^{d|V|}\rightarrow \mathbb{R}^{|E|}, ~ p \mapsto \left(\| p(v)-p(w)\|^2 \right)_{vw\in E}.
\end{equation*}
Given the 1-dimensional rigidity map $f_{G,1}$, 
we define the $|E| \times (d-k)|V|+k$ matrix 
\begin{align*}
    DR_k(G,p) := 
    \begin{bmatrix}
        R(G,\tilde{p}) & f_{G,1}(p_{d-k+1}) & \cdots & f_{G,1}(p_d)
    \end{bmatrix}.
\end{align*}

This matrix can be used to determine infinitesimal $(d,k)$-rigidity.
\begin{example}
Let $d=2$, $k=1$, $G=K_3$ and $V(K_3)=\{v_0,v_1,v_2\}$. Put $p(v_i)=(x_i,y_i)$ for $0 \leq i \leq 2$. Then we have 
\[ DR_1(K_3,p) = 
\begin{pmatrix} 
x_0-x_1 & x_1-x_0 & 0 &(y_0-y_1)^2  \\
x_0-x_2 & 0 & x_2-x_0 &(y_0-y_2)^2  \\
0 & x_1-x_2 & x_2-x_1 &(y_1-y_2)^2  
\end{pmatrix}. \]
\end{example}
\begin{theorem}\label{t:matrixrank}
    Let $(G,p)$ be a framework with a vertex $v_0 \in V$ where $p_i(v_0) \neq 0$ for each $i \in \{d-k+1, \ldots, d\}$.
    Then $(G,p)$ is infinitesimally $(d,k)$-rigid if and only if 
    \begin{equation*}
        \rank DR_k(G,p) = (d-k)|V| - \binom{d-k+1}{2} + k,
    \end{equation*}
    or $G$ is a complete graph and the set $\{p(v) : v \in V\}$ has affine dimension $\min \{ d, |V|-1\}$.
\end{theorem}

\begin{proof}
	We may suppose that $G$ is not a complete graph.	
	By multiplying each row of $J_{v_0}(G,p)$ corresponding to the vertex $v \in V \setminus \{v_0\}$ and coordinate $i \in \{d-k+1, \ldots, d\}$ by $p_i(v_0)$,
	we obtain the matrix
	\begin{align*}
	    J' =
	    \begin{bmatrix}
	        R(G,\tilde{p}) & R(G,p_{d-k+1}) & \cdots & R(G,p_{d}) \\
	        \mathbf{0} & M'_{d-k+1} &   & \mathbf{0} \\
	        \vdots &       &  \ddots    &      \\
	        \mathbf{0} & \mathbf{0}  &  & M'_d
	    \end{bmatrix},
	\end{align*}
	where
	\begin{align*}
	    M'_i(v,w) =
	    \begin{cases}
	        1 &\text{if $w = v$},\\
	        -p_i(v)/p_i(v_0) &\text{if $w = v_0$},\\
	        0 &\text{otherwise}.
	    \end{cases}
	\end{align*}
	Let $J'_{vw}$ be the row corresponding to the edge $vw \in E$,
	and let $J'_{v,i}$ be the row corresponding to the vertex $v \in V \setminus \{v_0\}$ and coordinate $i \in \{d-k+1, \ldots, d\}$.
	We will now form a new matrix from $J'$ by the following row operations:
	\begin{itemize}
	    \item For each row $vw$ with $v \neq 0$ and $w \neq v_0$,
	    $J'_{vw} \mapsto J'_{vw} - \sum_{i=d-k+1}^d (p_i(v) - p_i(w)) (J'_{v,i} - J'_{w,i})$.
	    \item For each row $v v_0$,
	    $J'_{vv_0} \mapsto J'_{v v_0} - \sum_{i=d-k+1}^d (p_i(v) - p_i(v_0)) J'_{v,i}$.
	    \item Shift each column corresponding to the vertex $v_0$ and coordinate $i \in \{d-k+1, \ldots, d\}$ to the right and multiply by $p_i(v_0)$.
	\end{itemize}
	With this we obtain the following matrix:
	\begin{align*}
	    J'' =
	    \begin{bmatrix}
	        R(G,\tilde{p}) & \mathbf{0} & \cdots & \mathbf{0} & f_{G,1}(p_{d-k+1}) & \cdots & f_{G,1}(p_d) \\
	        \mathbf{0} & I_{| V \setminus \{v_0\}|} &   & \mathbf{0} & b_{d-k+1} &   & \mathbf{0} \\
	        \vdots &       &  \ddots    &     &       &  \ddots    &       \\
	        \mathbf{0} & \mathbf{0}  &  & I_{| V \setminus \{v_0\}|} & \mathbf{0}  &  & b_d
	    \end{bmatrix},
	\end{align*}
	where $I_{| V \setminus \{v_0\}|}$ is the $| V \setminus \{v_0\}| \times | V \setminus \{v_0\}|$ identity matrix, 
	and each $b_i$ is the $| V \setminus \{v_0\}|$-dimensional column vector with coordinates $b_i(v) := -p_i(v)$.
	We now note that 
	\begin{align*}
		\rank DR_k(G,p) = \rank J'' - k(|V|-1) = \rank J_{v_0}(G,p) - k|V| +k
	\end{align*}
	which gives the desired equality.
\end{proof}

We next give a complete description of $(d,k)$-rigidity for arbitrary pairs $d,k$ in terms of the usual bar-joint rigidity. The characterisation leads to efficient combinatorial algorithms whenever the resulting bar-joint rigidity problem can be solved in such terms; that is, when $d-k\leq 2$. 

We will use an inductive argument to prove the following theorem. To this end it is convenient here, and only here, to allow $k=0$ in our definitions. In this case, $(d,0)$-rigidity is precisely $d$-rigidity.

\begin{theorem}\label{t:dkrigidgraph}
    A graph $G=(V,E)$ is $(d,k)$-rigid if and only if either $G$ is a complete graph, or $G$ contains a spanning $(d-k)$-rigid subgraph and $|E| \geq (d-k)|V| - \binom{d-k+1}{2} + k$.
\end{theorem}

\begin{proof}
	We may suppose that $G$ is not a complete graph.
	If $|E| < (d-k)|V| - \binom{d-k+1}{2} + k$,
	then $G$ is not $(d,k)$-rigid by \Cref{t:matrixrank}. 
	If $G$ does not contain a spanning $(d-k)$-rigid subgraph,
	then for every realisation $p :V \rightarrow \mathbb{R}^d$,
	the matrix $R(G,\tilde{p})$ has a rank strictly less than $(d-k)|V| - \binom{d-k+1}{2}$.
	Since $DR_k(G,p)$ is formed from $R(G, \tilde{p})$ by adding $k$ columns,
	it then must have a rank strictly less than $(d-k)|V| - \binom{d-k+1}{2} +k$.
    Hence $G$ is not $(d,k)$-rigid by \Cref{t:matrixrank}.

    Now suppose that $G$ contains a spanning $(d-k)$-rigid subgraph and $|E| \geq (d-k)|V| - \binom{d-k+1}{2} + k$.
	By deleting edges if necessary,
	we may suppose that $G = H + \{e_1, \ldots, e_k\}$,
	where $H=(V,F)$ is minimally $(d-k)$-rigid and $e_1, \ldots, e_k$ are edges in $E \setminus F$.
	Define $E_0 = F$, $E_i := F +\{e_1,\ldots,e_i\}$, $G_0 := (V,E_0)$ and $G_i := (V,E_i)$ for all $1\leq i \leq k$.
	It is immediate that $G_0 = H$ is minimally $(d-k,0)$-rigid.
	By an inductive argument,
	suppose that $G_j$ is minimally $(d-k+j,j)$-rigid for some $j \in \{0,\ldots, k-1\}$.
	By our inductive argument and \Cref{t:matrixrank},
	there exists a realisation $p : V \rightarrow \mathbb{R}^{d-k+j}$ such that
    \begin{equation*}
        |E_j| = (d-k)|V| - \binom{d-k+1}{2} + j.
    \end{equation*}
	By maximality,
	we have that $\rank DR_j(G_{j+1},p) =  \rank DR_j(G_j,p)$,
	hence there exists a unique (up to scalar multiple) element of the left kernel of $\rank DR_j(G_{j+1},p)$.
    This is equivalent to their existing a unique (up to scalar multiple) non-zero vector $\sigma : E_{j+1} \rightarrow \mathbb{R}$ where $\sigma^T R(G,\tilde{p}) = [0 ~ \ldots ~ 0]$ and 
	\begin{equation}\label{e:dkrigidgraph}
		\sum_{vw \in E_{j+1}} \sigma(vw) (p_i(v) - p_i(w))^2 = p_i^T \Omega(\sigma) p_i = 0
	\end{equation}
	for each $i \in \{d-k+1, \ldots, d-k+j\}$,
	where $\Omega(\sigma)$ is the stress matrix corresponding to $\sigma$.
	Since $\sigma$ is non-zero,
	there exists $z \in \mathbb{R}^V$ such that $z^T\Omega(\sigma) z \neq 0$.
	Fix $p' :V \rightarrow \mathbb{R}^{d-k+j+1}$ to be the realisation where $p'_i := p_i$ for all $i \in \{1, \ldots, d-k+j\}$,
	and $p'_{d-k+j+1} := z$.
	As the left kernel of $DR_{j+1}(G_{j+1},p')$ is contained in the left kernel of $DR_{j}(G_{j+1},p)$,
	we have $\ker DR_{j+1}(G_{j+1},p')^T = \{0\}$.
	By counting edges we see that $|E| = (d-k)|V| - \binom{d-k+1}{2} +j+1$,	
	hence $G_{j+1}$ is minimally $(d-k+j+1,j+1)$-rigid.
	By induction it now follows that $G = G_k$ is minimally $(d,k)$-rigid.
\end{proof}

\begin{example}
Let us unpack \Cref{t:dkrigidgraph} for some basic special cases.
Combining with the folklore 1-dimensional characterisation of rigidity, the theorem implies that a graph is $(2,1)$-rigid if and only if either it is complete on $1$ or $2$ vertices or it is a connected graph with at least one cycle.
Similarly, a graph $G=(V,E)$ is $(3,2)$-rigid if and only if either $G$ is complete on at most 3 vertices or $G$ is connected with $|E|\geq |V|+1$.

Next suppose the gap between $d$ and $k$ is 2. Then we can use Laman's theorem \cite{laman,PG} to deduce that a graph is $(3,1)$-rigid if and only if $G$ is complete or it is a Laman-plus-one graph (that is, it is obtained from a Laman graph\footnote{$G=(V,E)$ is a Laman graph if $|E|=2|V|-3$ and every edge-induced subgraph $(V',E')$ has $|E'|\leq 2|V'|-3$.} by adding exactly one edge). 

When the gap is bigger than 2 we no longer have a combinatorial description of rigidity to rely upon. For example, a graph $G=(V,E)$ is $(4,1)$-rigid if and only if either $G$ is complete or $G$ contains a spanning 3-rigid subgraph and $|E|\geq 3|V|-5$. Nevertheless some $d$-rigidity is understood in some special cases which we can then apply. For example if $G$ is obtained from a triangulation of the sphere by adding some edges then, combining our result with a theorem of Gluck \cite{Glu}, we have that $G$ is $(4,1)$-rigid.
\end{example}

\begin{remark}
A ($d$-dimensional) \emph{0-extension} adds a vertex of degree $d$ to a graph. A ($d$-dimensional) \emph{1-extension} deletes an edge $xy$ and adds a vertex $v$ of degree $d+1$ incident to $x$ and $y$. It is well known that these operations preserve the rigidity of bar-joint frameworks \cite{Wlong}.
We note that it is possible to extend the standard ($(d-k)$-dimensional) 0- and 1-extension arguments to show that $(d,k)$-rigidity is preserved by these operations. This gives a way to construct large families of $(d,k)$-rigid graphs. It also gives a combinatorial proof of an interesting special case of \Cref{t:dkrigidgraph} when $d=3$ and $k=1$. Here one may use Polaczek-Geiringer's \cite{PG} characterisation of 2-rigidity to see that the characterisation of minimal $(3,1)$-rigidity in \Cref{t:dkrigidgraph} is equivalent to the graph being a Laman-plus-one graph. Hence to prove such graphs are $(3,1)$-rigid we simply apply a well known recursive construction of Laman-plus-one graphs due to Haas et al \cite{Haas}.
\end{remark}

The $(d,k)$-rigidity matroid $\mathcal{R}_{d,k}(G,p)$ of a framework $(G,p)$ is the row matroid of the Jacobean matrix $J_{v_0}(G,p)$. If $(G,p)$ and $(G,q)$ are generic frameworks in $\mathbb{R}^d$, then it is easy to see that their matroids coincide, that is the $(d,k)$-rigidity matroid depends on $d,k$ and $G$ but, for generic frameworks, it does not depend on the choice of generic realisation. Hence we drop the $p$ and use $\mathcal{R}_{d,k}(G)$ or even $\mathcal{R}_{d,k}$ when the context is clear. We will also use $\mathcal{R}_{d}$ for the usual $d$-dimensional rigidity matroid of a graph and $U_k$ for the uniform matroid of rank $k$.

Let $M_1$ and $M_2$ be two matroids with common ground set $E$. Then the \emph{matroid union} $M_1 \vee M_2$ is the matroid on $E$ with the property that a subset $F$ is independent in $M_1 \vee M_2$ if and only if it has the form $F=F_1\cup F_2$, where $F_i$ is independent in $M_i$ for $i=1,2$.

\begin{corollary}\label{cor:union}
We have $\mathcal{R}_{d,k}= \mathcal{R}_{d-k} \vee U_k$.
\end{corollary}

\begin{proof}
It suffices to show that $F$ is a basis in $\mathcal{R}_{d,k}$ if and only if we can partition $F$ into sets $F_1$ and $F_2$ where $F_1$ is a basis of $\mathcal{R}_{d-k}$ and $F_2$ is a basis of $U_k$ (equivalently $F_2$ has size $k$). This follows from \Cref{t:dkrigidgraph}, since $F$ is a basis of $\mathcal{R}_{d,k}$ if and only if the graph induced by $F$ is minimally $(d,k)$-rigid and $F_1$ is a basis of $\mathcal{R}_{d-k}$ if and only if the graph induced by $F_1$ is minimally $(d-k)$-rigid.
\end{proof}

It follows from a well known algorithm of Edmond's \cite{Edm}, see for example a similar discussion in \cite[Section 5.1]{DK}, that generic $(d,k)$-rigidity can be checked efficiently whenever $(d-k)$-rigidity can.

\begin{remark}
    There are types of graph rigidity where the associated rigidity matroid is the matroid union of a rigidity matroid and some other matroid.
    In \cite{capr},
    Cros et al. investigated the rigidity of unmanned aerial vehicles where there exists a time lag between distance information being sent between any two vehicles.
    They proved that the correct rigidity matroid in this setting (with all vehicles existing in $d$-dimensional space) was the matroid union of the $(d-1)$-dimensional rigidity matroid and the graphical matroid.
    This type of rigidity matroid is identical to that found by Dewar and Kitson in \cite{DK}. We also note that Schulze, Serocold and Theran \cite{SST} considered rigidity in the context where classes of edges can change in a specific coordinated way and also found a matroid union structure, there with the transversal matroid.
\end{remark}

\section{Global \texorpdfstring{\lowercase{$(d,k)$}}{\lowercase{(d,k)}}-rigidity}
\label{sec:global}

We next consider global $(d,k)$-rigidity.
As noted in the proof of \Cref{t:dkrigidgraph},
a map $\sigma : E \rightarrow \mathbb{R}$ with corresponding stress matrix $\Omega(\sigma)$ is an element of $\ker DR_k(G,p)^T$ if and only if it is an equilibrium stress of $(G,\tilde{p})$ and $p_i^T \Omega(\sigma) p_i =0$ for each $i \in \{d-k+1,\ldots, d\}$.

\begin{example}\label{ex:k4-e}
In \Cref{ex:c4} we showed that $C_4$ is minimally $(2,1)$-rigid. From this, it is not hard to then determine that $C_4+v_1v_3=K_4-v_2v_4$ is a circuit in $\mathcal{R}_{2,1}$. Hence, any generic realisation has a unique equilibrium stress which is non-zero on each edge. If we fix $p$ to be the realisation of $K_4-v_2v_4$ with $p(v_2)=(1,2)$, $p(v_3)=(6,8)$ and $p(v_4)=(16,12)$, we have the unique element $\sigma = \begin{bmatrix} 490 & 98 & -8 & 5 & -95 \end{bmatrix}^T \in \ker DR_1(K_4-v_2v_4,p)^T$ with corresponding stress matrix
$$\Omega(\sigma)=\begin{bmatrix} 
400 & -490 & 95 & -5 \\
-490 & 588 & -98 & 0 \\
95 & -98 & -5 & 8 \\
-5 & 0 & 8 & -3
\end{bmatrix},$$
In this case we have $\rank \Omega = 4-2+1-1=2$.
\end{example}

\subsection{A sufficient condition}

We next develop a sufficient condition for global $(d,k)$-rigidity (\Cref{t:dkgr1}).
We first prove the following results surrounding equilibrium stresses.

\begin{lemma}\label{l:samestresses}
	For a graph $G=(V,E)$ with fixed vertex $v_0$, choose vectors $\sigma  \in \mathbb{R}^E$ and $\lambda \in \mathbb{R}^{V\setminus\{v_0\}}$.
	For any integers $1 \leq k <d$,
	let $p:V \rightarrow \mathbb{R}^d$ be a realisation of $G$ where $p_i(v_0) \neq 0$ for all $i \in \{d-k+1,\ldots,d\}$.
	Then $(\sigma, \lambda) \in \ker J_{v_0}(G,p)^T$ if and only if $\sigma \in DR_k(G,p)$ and
	\begin{equation}\label{eq:lambda}
		\lambda(v) = -p_i(v_0)\left( \sum_{w \in N_G(v)} \sigma(vw) (p_i(v) - p_i(w)) \right)
	\end{equation}
	for each $i \in \{d-k+1,\ldots,d\}$.
\end{lemma}

\begin{proof}
	We first note that in either direction of the implication we require that $\sigma$ is an equilibrium stress of $(G,\tilde{p})$,
    so we may suppose that this is so throughout the proof.
	Label the entry of the vector $J_{v_0}(G,p)^T(\sigma, \lambda)$ that corresponds to a vertex $v \in V$ and coordinate $i \in \{1,\ldots, d\}$ by $a(v,i)$.
    As $\sigma$ is an equilibrium stress of $(G,\tilde{p})$,
    it follows that for every $i \in \{1,\ldots,d-k\}$ and every for every $v \in V$ we have
    \begin{equation*}
        a(v,i) = \sum_{w \in N_G(v)} \sigma(vw) (p_i(v)-p_i(w)) = 0.
    \end{equation*}
    Fix $i \in \{d-k+1,\ldots, d\}$.
    If $v \neq v_0$ then
	\begin{equation*}
		a(v,i) = \frac{\lambda(v)}{p_i(v_0)} + \sum_{w \in N_G(v)} \sigma(vw) (p_i(v)-p_i(w))
    \end{equation*}
    and $a(v,i) = 0$ if and only if \cref{eq:lambda} holds.
	Given $\Omega(\sigma)$ is the corresponding stress matrix for $\sigma$,
	we have that
	\begin{align*}
		a(v_0,i) &= \sum_{w \in N_G(v_0)} \sigma(v_0w) (p_i(v_0)-p_i(w)) - \sum_{v \neq v_0} \frac{\lambda(v)p_i(v)}{p_i(v_0)^2} \\
		&=\frac{1}{p_i(v_0)} \left( \sum_{v \in V} \sum_{w \in N_G(v)} \sigma(v_0w) p_i(v) (p_i(v)-p_i(w)) \right) \\
		&=\frac{p_i^T \Omega(\sigma) p_i}{p_i(v_0)}.
	\end{align*}
    Hence, $a(v_0,i) = 0$ if and only if $\sigma^T f_{G,1}(p_i)$.
	The result follows from combining the above observations.
\end{proof}

\begin{lemma}\label{l:genstresses}
	Let $(G,p)$ be a generic $(d,k)$-rigid framework in $\mathbb{R}^d$ and let $(G,q)$ be a $(d,k)$-equivalent framework.
	Then $\ker DR_k(G,p)^T =\ker DR_k(G,q)^T$ and $\ker J_{v_0}(G,p)^T =\ker J_{v_0}(G,q)^T$ for any choice of $v_0 \in V$.
	Furthermore, if $\sigma \in \ker DR_k(G,p)^T$, $i \in \{d-k+1,\ldots, d\}$ and $q_i = \alpha_i p_i$,
	then
	\begin{equation}\label{e:genstresses}
		(1-\alpha_i^2)\left( \sum_{w \in N_G(v)} \sigma(vw) (p_i(v) - p_i(w)) \right) =0.
	\end{equation}
\end{lemma}

To prove the lemma we require a proposition from Connelly \cite{C2005}.

\begin{proposition}\label{prop:bob}
Suppose that $f : \mathbb{R}^a \rightarrow \mathbb{R}^b$ is a function, where each coordinate is a rational function\footnote{This extension from Connelly's statement for polynomial functions can be proved with the same technique as in \cite{C2005}.} with integer coefficients, $p \in \mathbb{R}^a$ is generic with $f(p)$ well-defined, and $f(p) = f(q)$, for some $q \in \mathbb{R}^a$. Then there are (open) neighbourhoods $N_p$ of $p$ and $N_q$ of $q$ in $\mathbb{R}^a$ and a diffeomorphism $g:N_q \rightarrow N_p$ such that for all $x\in N_q$, $f(g(x))=f(x)$, and $g(q)=p$. 
\end{proposition} 

Let\footnote{Here we use a dashed arrow to represent that $F$ is not well-defined at all points.} $F:\mathbb{R}^{d|V|}\dashrightarrow \mathbb{R}^{|E|+|V|-1}$ be the concatenation of the $d$-dimensional rigidity map $f_{G,d}$
and the \emph{dilation map} $D:\mathbb{R}^{d|V|}\dashrightarrow \mathbb{R}^{|V|-1}$ such that $D(p)=(\dots, (\frac{p_i(v)}{p_i(v_0)}), \dots)_{v\in V\setminus\{v_0\}}$. Then $J_{v_0}(G,p)$ is the Jacobean of $F$ evaluated at $p\in \mathbb{R}^{d|V|}$.

\begin{proof}[Proof of \Cref{l:genstresses}]
    As $p$ is generic and $q$ is equivalent to $p$,
    both $F(p)$ and $F(q)$ are well-defined.
    By applying \Cref{prop:bob} to the map $F$,
    there exist open neighbourhoods $N_p$ of $p$ and $N_q$ of $q$ in $\mathbb{R}^{d|V|}$ and a diffeomorphism $g : N_q \rightarrow N_p$ such that $g(q) = p$ and, for all $q \in N_q$, $F(g(q)) = F(q)$. Taking differentials at $q$,
    we obtain $J_{v_0}(G,q)=J_{v_0}(G,p)L$
    where $L$ is the Jacobian matrix of $g$ at $q$. If $(\sigma, \lambda) \in \ker J_{v_0}(G,p)^T$ we have 
    $$J_{v_0}(G,q)^T (\sigma, \lambda) = L^TJ_{v_0}(G,p)^T (\sigma, \lambda) = L^T \mathbf{0}  = \mathbf{0}.$$ 
    Hence $(\sigma,\lambda) \in \ker J_{v_0}(G,q)^T$.
    It now follows by symmetry that
    $\ker J_{v_0}(G,p)^T =\ker J_{v_0}(G,q)^T$.
    Hence, by \Cref{l:samestresses}, we have $\ker DR_k(G,p)^T =\ker DR_k(G,q)^T$.
    Furthermore, given our previous choice of $\sigma$ and any $i \in \{d-k+1,\ldots, d\}$, we have 
    \begin{align*}
        p_i(v_0)\left( \sum_{w \in N_G(v)} \sigma(vw) (p_i(v) - p_i(w)) \right) = q_i(v_0)\left( \sum_{w \in N_G(v)} \sigma(vw) (q_i(v) - q_i(w)) \right).
    \end{align*}
    We now obtain \cref{e:genstresses} from the above equation by substituting $q_i=\alpha_i p_i$, dividing both sides by $p_i(v_0)$ and rearranging.
\end{proof}

\Cref{l:genstresses} is a key step in obtaining an analogue of the sufficiency direction of the following central theorem in rigidity theory.

\begin{theorem}[Connelly \cite{C2005}; Gortler, Healy and Thurston \cite{GHT}]\label{t:gr}
	Let $(G,p)$ be a generic framework in $\mathbb{R}^d$.
	Then $(G,p)$ is globally $d$-rigid if and only if there exists an equilibrium stress $\omega$ of $(G,p)$ such that $\rank \Omega(\omega)=|V|-d-1$.
	Furthermore,
	if $(G,q)$ is also a generic framework in $\mathbb{R}^d$,
	then $(G,q)$ is globally $d$-rigid if and only if $(G,p)$ is globally $d$-rigid.
\end{theorem}

We will prove that a similar stress rank condition is sufficient for global $(d,k)$-rigidity.

\begin{theorem}\label{t:dkgr1}
	Let $(G,p)$ be a generic framework in $\mathbb{R}^d$.
	If there exists $\sigma \in \ker DR_k(G,p)^T$ such that $\rank \Omega(\sigma)=|V|-d+k-1$,
	then $(G,p)$ is globally $(d,k)$-rigid.
\end{theorem}

\begin{proof}
	Fix a vertex $v_0 \in V$ and choose any $(d,k)$-equivalent framework $(G,q)$.
	By \Cref{l:genstresses},
	\cref{e:genstresses} holds for each $i \in \{d-k+1,\ldots, d\}$.
	If $\sum_{w \in N_G(v)} \sigma(vw) (p_i(v) - p_i(w)) = 0$ for some $i \in \{d-k+1,\ldots,d\}$ then $\rank \Omega(\sigma)< |V| - d+k-1$ (since its kernel already contains the all 1's vector and $p_1,\ldots,p_{d-k}$),
	hence $\alpha_i = \pm 1$ for each $i \in \{d-k+1,\ldots, d\}$.
	By applying reflections to $q$, we may suppose that $\alpha_i=1$, and hence $q_i = p_i$, for all $i \in \{d-k+1, \ldots,d\}$.
	Hence for each $vw \in E$ we have
	\begin{align*}
		\|p(v) - p(w)\|^2 - \|q(v)-q(w)\|^2 &= \|\tilde{p}(v) - \tilde{p}(w)\|^2 - \|\tilde{q}(v)-\tilde{q}(w)\|^2
	\end{align*}
	and thus $(G,\tilde{p})$ is equivalent to $(G,\tilde{q})$.
	Since $\sigma$ is an equilibrium stress of $(G,\tilde{p})$,
	$\tilde{q}$ is congruent to $\tilde{p}$ by \Cref{t:gr}.	
	If we choose any two (possibly non-adjacent) vertices $v,w \in V$, then
	\begin{align*}
		\|q(v)-q(w)\|^2 &= \|\tilde{q}(v)-\tilde{q}(w)\|^2 + \sum_{i=d-k+1}^d (q_i(v) - q_i(w))^2 \\
		&=\|\tilde{p}(v)-\tilde{p}(w)\|^2 + \sum_{i=d-k+1}^d (p_i(v) - p_i(w))^2 \\
		&= \|p(v) - p(w)\|^2,
	\end{align*}
	hence $q$ is congruent to $p$.
\end{proof}

We next show that we can use the theorem to recursively construct a large class of globally $(d,k)$-rigid graphs.
The proof techniques of the following two lemmas are sufficiently well used in the literature that we provide only sketches.

\begin{lemma}\label{lem:1extstress}
Suppose $(G, p)$ is a generic framework in $\mathbb{R}^d$ on at least $d-k+1$ vertices. Let $G' = (V', E')$ be obtained from $G$ by deleting an edge $e = v_1v_2$ and adding a new vertex $v_0$ and new edges $v_0v_1, v_0v_2, \dots, v_0v_{d-k+1}$. Then there exists a map $q : V'\rightarrow \mathbb{R}^d$ such that $\rank J_{v_0}(G',q) = \rank J_{v_0}(G,p)+d$. Furthermore, if $\sigma$ is an equilibrium stress for $(G,p)$ with corresponding stress matrix $\Omega(\sigma)$ and $\sigma_e\neq 0$, then there exists an equilibrium stress $\sigma'$ for $(G', q)$ such that $\rank \Omega(\sigma')=\rank \Omega(\sigma) +1$.
\end{lemma}

\begin{proof}[Sketch of proof]
Define $(G', q)$ by putting $q(v) = p(v)$ for all $v \in V$ and $q(v_0) = \frac{1}{2}( p(v_1) +  p(v_2))$. It is now straightforward to use the standard ``collinear triangle'' technique (see, for example, \cite{Wlong}) to show that $\rank J_{v_0}(G',q) = \rank J_{v_0}(G,p)+d$.

Let $\sigma'$ be the equilibrium stress for $(G', q)$ defined by putting $\sigma'_f = \sigma_f$ for all 
$f \in E-e$, $\sigma'_{v_0v_1} = 2\sigma_e$, $\sigma'_{v_0v_2} = 2\sigma_e$
and $\sigma_f=0$ otherwise. It is straightforward to verify that $\sigma'$ is an equilibrium stress. 
We may now manipulate the stress matrix $\Omega(\sigma')$ to see that $\rank \Omega(\sigma') = \rank \Omega(\sigma) +1$ (see \cite[Lemma 6.1]{JNstress} for the same argument in a different context).
\end{proof}

\begin{lemma}\label{lem:perturbstress}
Suppose $(G,p)$ is a $(d,k)$-rigid framework in $\mathbb{R}^d$ and  
$\rank \Omega = |V|-d + k-1$ for some equilibrium stress $\sigma$ of $(G, p)$.
Then $(G, q)$ has an equilibrium stress $\sigma'$ with $\rank \Omega' = |V|-d+k-1$ for all generic $q \in \mathbb{R}^{d|V|}$. In addition, $\sigma'$ can be chosen so that $\sigma'_e\neq 0$ for all $e \in E$. 
\end{lemma}

\begin{proof}[Sketch of proof]
The first conclusion uses the technique of Connelly and Whiteley \cite[Theorem 5]{CW}.
In essence,
we can define function $f$ from the set of frameworks with maximal rank $(d,k)$-rigidity matrix to the space $\mathbb{R}^E$ such that $f(q)$ is an equilibrium stress of $(G,q)$, each entry of $f$ is a rational function, and $f(p) = \omega$.
By interpreting a stress matrix having full rank as an algebraic condition regarding determinants one can see that every generic framework must also have a full rank stress matrix.

The second conclusion follows from the same argument as in \cite[Lemma 6.3]{JNstress}. Briefly, suppose $\sigma'_e=0$ for some $e\in E$. Then $\sigma'$ is an equilibrium stress of the $(d,k)$-rigid framework $(G-e,p)$. Hence there is another equilibrium stress $\sigma^*$ obtained by adding $e$ which is non-zero on $e$. Then for sufficiently small $\varepsilon$ the matrix $\Omega(\sigma'+\varepsilon \sigma^*)$ will have the same rank as $\Omega(\sigma')$.
\end{proof}

\begin{corollary}\label{cor:1extglob}
Let $G=(V,E)$, let $(G,p)$ be a generic framework in $\mathbb{R}^d$ and suppose there exists $\sigma \in \ker DR_k(G,p)^T$ such that $\sigma_e\neq 0$ for all $e\in E$ and $\rank \Omega(\sigma)=|V|-d+k-1$. Let $H$ be a graph obtained from $G$ by a sequence of $(d-k)$-dimensional 1-extensions and edge additions. Then any generic framework $(H,q)$ is globally $(d,k)$-rigid.    
\end{corollary}

\begin{proof}
By \Cref{t:dkgr1} it now suffices to show that a sequence of $(d-k)$-dimensional 1-extensions and edge additions preserves full rank stress matrix. This in turn follows from \Cref{lem:1extstress} and the fact that we may simply choose $\sigma_e=0$ for any edge addition $e$ provided that we can show that any edge $f$ we perform a 1-extension on has $\sigma_f\neq 0$.
This was proved in \Cref{lem:perturbstress}.
\end{proof}

\subsection{A necessary condition}

We also consider whether the following natural necessary conditions for global rigidity can be adapted.

\begin{theorem}[Hendrickson \cite{hendrickson}]\label{t:hendrickson}
    Let $(G,p)$ generic framework in $\mathbb{R}^d$.
    If $(G,p)$ is globally $d$-rigid then either $|V(G)|\leq d+1$ and $G$ is complete or $|V(G)|\geq d+2$
    and $G$ is $(d+1)$-connected and $G-e$ is $d$-rigid for any edge $e$ of $G$.
\end{theorem}

We can obtain a weak version of a similar result for arbitrary frameworks by linking global $(d,k)$-rigidity to global $(d-k)$-rigidity. In the generic case we suspect a stronger statement may be possible.

\begin{proposition}\label{p:dkhendrickson}
    Let $(G,p)$ a framework in $\mathbb{R}^d$.
    If $(G,p)$ is globally $(d,k)$-rigid and $G$ is not complete,
    then $(G,\tilde{p})$ is globally $(d-k)$-rigid.
\end{proposition}

\begin{proof}
	Suppose that $(G,\tilde{p})$ is not globally $(d-k)$-rigid.
	Then there exists a realisation $\tilde{q}:V \rightarrow \mathbb{R}^{d-k}$ of $G$ where $(G,\tilde{q})$ is equivalent to, but not congruent to, $(G,\tilde{p})$.
	Define $q :V \rightarrow \mathbb{R}^d$ to be the realisation of $G$ where $q(v) = (\tilde{q}(v),p_{d-k+1}(v),\ldots, p_d(v))$ for each $v \in V$.
	For each pair $v,w \in E$ we have
	\begin{align*}
		\|p(v) - p(w)\|^2 - \|q(v)-q(w)\|^2 &= \|\tilde{p}(v) - \tilde{p}(w)\|^2 - \|\tilde{q}(v)-\tilde{q}(w)\|^2,
	\end{align*}
	hence $(G,q)$ is $(d,k)$-equivalent to, but not congruent to, $(G,p)$,
	i.e., $(G,p)$ is not globally $(d,k)$-rigid.
\end{proof}

Observe that the resulting necessary connectivity condition for global $(d,k)$-rigidity, by depending on $k$, can be much lower than that needed for global $d$-rigidity. The following simple lemma illustrates that globally $(d,k)$-rigid graphs need not be highly connected.

\begin{lemma}\label{lem:unionglob}
Let $G_1$ and $G_2$ be graphs on at least $d-k+2$ vertices and with at least $d-k+1$ vertices in common. Let $(G,p)$ be a generic realisation of $G = G_1 \cup G_2$ and let $p^i = p|_{G_i}$. Suppose that $(G_i, p^i)$ is globally $(d,k)$-rigid for $i = 1, 2$. Then $(G, p)$ is globally $(d,k)$-rigid.   
\end{lemma}

\begin{proof}
Let $(G, q)$ be a $(d,k)$-equivalent framework to $(G, p)$. By applying a suitable isometry to $q$ we may assume that $p(u) = q(u)$ for all $u\in V (G_1) \cap V (G_2)$. Since $(G_1,p^1)$ is globally $(d,k)$-rigid, $q|_{G_1} = \iota \circ p^1$ for some isometry $\iota$. Since $(G_2, p^2)$ is globally $(d,k)$-rigid, there is a unique equivalent realisation of $G_2$ which maps $u$ to $p(u)$ for all $u\in V (G_1) \cap V (G_2)$. Since both $(G_2,q|_{G_2})$ and $(G_2,\iota \circ p^2)$ have this property, $q|_{G_2} =\iota \circ p^2$. Hence $q=\iota \circ p$ and $(G, p)$ is congruent to $(G, q)$.    
\end{proof}

\section{Concluding remarks}
\label{sec:conclude}

We conclude the paper with two conjectures on generic global $(d,k)$-rigidity. The first is that a full stress rank is necessary as well as sufficient, analogous to \Cref{t:gr}.

\begin{conjecture}\label{t:dkgr2}
	Let $(G,p)$ be a generic framework in $\mathbb{R}^d$.
	If $(G,p)$ is globally $(d,k)$-rigid and $G$ is not complete,
	then there exists $\sigma \in \ker DR_k(G,p)^T$ such that $\rank \Omega(\sigma) =|V|-d+k-1$.
\end{conjecture}

We know from \Cref{p:dkhendrickson} and \Cref{t:gr} that if $(G,p)$ is globally $(d,k)$-rigid,
then there exists an equilibrium stress $\sigma$ of $(G, \tilde{p})$ such that $\rank \Omega(\sigma)=|V|-d+k-1$.
For $\sigma \in \ker DR_k(G,p)^T$,
we also require that $p_i^T \Omega p_i =0$ for all $i \in \{d-k+1,\ldots,d\}$. It seems non-trivial to prove this.

We conclude the paper with an alternative conjectured characterisation motivated by the characterisation in \cite{J&J}.

\begin{conjecture}
	Suppose positive integers $d,k$ are chosen so that $d-k\leq 2$ and $(G,p)$ is generic.
 Then the following are equivalent:
 \begin{enumerate}
    \item \label{conj1} $(G,p)$ is globally $(d,k)$-rigid;
    \item \label{conj2} there exists a set $F$ of $k$ edges such that $(G-F,\tilde p)$ is $(d-k+1)$-connected and redundantly $(d-k)$-rigid; and
    \item \label{conj3} there exists a set $F$ of $k$ edges such that $(G-F,\tilde p)$ is globally $(d-k)$-rigid.
 \end{enumerate}
\end{conjecture}

Since $d-k\leq 2$, by a folklore result and \cite{J&J}, \ref{conj2} and \ref{conj3} are equivalent. We suspect that \ref{conj1} implies \ref{conj3} could be proved by a Hendrickson-type argument \cite{hendrickson}.
If this were true it would suffice to prove that \ref{conj2} implies \ref{conj1}. We verify this in the case when $d=2$ and $k=1$. 

\begin{lemma}
Let $(G,p)$ be generic.
Suppose $G$ contains an edge $e$ such that $G-e$ is 2-connected.  
Then $(G,p)$ is globally $(2,1)$-rigid.
\end{lemma}

\begin{proof}
Let $H$ denote the graph obtained from $K_4$ by deleting an edge. 
It is easy to see that every graph $G$ such that $G-e$ is 2-connected can be obtained from $H$ by subdividing edges and adding new edges. In \Cref{ex:k4-e} we constructed a nowhere zero full rank stress for a specific $(d,k)$-rigid realisation of $H$. Hence the result follows from \Cref{lem:perturbstress} and \Cref{cor:1extglob} by induction. 
\end{proof}

\section*{Acknowledgements}

S.\,D.\ was supported by the Heilbronn Institute for Mathematical Research.  A.\,N.\ was partially supported by EPSRC grant EP/X036723/1.

\bibliographystyle{abbrv}
\def\lfhook#1{\setbox0=\hbox{#1}{\ooalign{\hidewidth
  \lower1.5ex\hbox{'}\hidewidth\crcr\unhbox0}}}

\end{document}